\documentclass[12pt]{amsart}
\usepackage[all]{xy}
\usepackage{pgf}
\usepackage{tikz}
\usepackage{tikz-cd}
\usepackage[margin=1in]{geometry}  
\usepackage{graphicx}              
\usepackage{amsmath}               
\usepackage{amsfonts}              
\usepackage{amsthm}                
\usepackage{euscript}

\theoremstyle{definition}
\newtheorem{thm}{Theorem}[section]

\newtheorem{prop}[thm]{Proposition}
\newtheorem{cor}[thm]{Corollary}

\newtheorem{rem}[thm]{Remark}
\newtheorem{defn}{Definition}[section]
\newtheorem{exa}[thm]{Example}
\newtheorem{question}[thm]{Question}

\DeclareMathOperator{\op}{op}
\DeclareMathOperator{\Sym}{Sym}
\DeclareMathOperator{\pt}{pt}
\DeclareMathOperator{\tst}{tst}
\DeclareMathOperator{\st}{st}

\DeclareMathOperator{\GL}{GL}

\DeclareMathOperator{\id}{id}

\DeclareMathOperator{\codim}{codim}
\DeclareMathOperator{\Hom}{Hom}
\DeclareMathOperator{\Spec}{Spec}

\newcommand{\RR}{\mathbb{R}}      
\newcommand{\ZZ}{\mathbb{Z}}      
\newcommand{\QQ}{\mathbb{Q}}

\newcommand{\tbf}{\textbf}

\newcommand{\A}{\mathbb{A}}

\newcommand{\X}{{\mathcal X}}

\newcommand{\Z}{{\mathcal Z}}

\newcommand{\cat}{\EuScript}    

\newcommand{\Sch}{{\cat Sch}}

\newcommand{\Ab}{{\cat Ab}}

\tikzset{node distance=2cm, auto}

\title{The cone theorem and the vanishing of Chow cohomology}

\author{Dan Edidin}
\author{Ryan Richey}
\address{Department of Mathematics, University of Missouri, Columbia MO 65211}
\email{edidind@missouri.edu}
\email{rrtf3@mail.missouri.edu}

\thanks{The first author was supported by Simons Collaboration Grant
  315460.}


\begin{document}

\begin{abstract}
  We show that a cone theorem for $\A^1$-homotopy invariant
  contravariant functors implies the vanishing of
  the positive degree part of the operational Chow cohomology rings
  of a large class of affine varieties. We also discuss how this vanishing
  relates to a number of questions about representing Chow cohomology classes
  of GIT quotients in terms of equivariant cycles. 
  \end{abstract}

\maketitle

\section{Introduction}
In \cite{FuMa:81,Ful:84}, Fulton and MacPherson define for any scheme $X$ a
graded cohomology ring $A^*_{\op}(X)$ which equals the classical
intersection ring when $X$ is non-singular. An element of
$A^*_{\op}(X)$ is a collection of operations on Chow groups of
$X$-schemes compatible with basic operations in intersection theory.
The product structure is given by composition. By construction there
is a pullback of operational rings $A^*_{\op}(Y) \to A^*_{\op}(X)$
for any morphism $X \to Y$. 

For an arbitrary singular scheme elements of $A^*_{\op}(X)$ do not
have natural interpretations in terms of algebraic cycles,
although Kimura \cite{Kim:92} showed how $A^*_{\op}(X)$ can be related to the
intersection ring of a resolution of singularities of $X$. 
Despite the formal structure, there is a class of singular varieties
where the operational Chow groups are readily computable. Specifically, if $X$
is a complete linear variety then Totaro \cite{Tot:14} proved that the
pairing
$$A^k_{\op}(X) \times A_k(X) \to A_0(X) = \ZZ\; , (c,\alpha)\mapsto c \cap \alpha$$
is perfect so $A^k_{\op}(X) = \Hom(A_k(X), \ZZ)$, where $A_k(X)$ denotes the classical Chow group of $k$-dimensional cycles. 
When $X$ is a complete toric variety Fulton and Sturmfels
\cite{FuSt:97} gave an explicit description of the operational
product in terms of Minkowski weights. This was generalized by Payne 
\cite{Pay:06} who showed that the $T$-equivariant operational Chow ring of
a toric variety $X$
can be identified with the ring of integral piecewise polynomial functions on the fan
of $X$. 

The Cox construction expresses any toric variety as a good quotient
$(\A^n \setminus B)/G$ where $G$ is a diagonalizable group. In
\cite{EdSa:17} the first author and Satriano showed that if $X = Z/G$
is the good quotient of a smooth variety by a linearly reductive group
then the rational operational Chow ring $A^*_{\op}(X)_\QQ$ naturally
embeds in the equivariant Chow ring $A^*_G(Z)_\QQ$. Moreover, the image
of an element in $A^k_{\op}(X)$ is represented by a class of the
form $\sum c_i [Z_i]$ where the $Z_i \subset Z$ are codimension-$k$
$G$-invariant subvarieties of $Z$ which are saturated with respect to
the quotient map $\pi \colon Z \to X$.

By the \'etale slice theorem \cite{Lun:73} the local model for the
good quotient of a smooth variety at a closed orbit $Gx \subset Z$ is
the quotient $V \to V/G_x$ where $V = T_{x,Z}/T_{x,Gx}$ is the normal
space to the orbit $Gx$ at $x$, and $G_x$ is the stabilizer of
$x$. Therefore, a natural problem is to compute the operational Chow
rings of good quotients $V/G$ where $V$ is a representation of a
linearly reductive group $G$.  Examples of such quotients are affine
toric varieties associated to maximal dimensional
strongly convex polyhedral cones.  In
\cite{Ric:19, EdRi:19} it is shown that $A^*(X) = \ZZ$ for any affine
toric variety $X$, and likewise that $\op K^0(X) =\ZZ$ where $\op K^0$
is the operational $K$-theory defined by Anderson and Payne
\cite{AnPa:15}.

The purpose of this paper is to show that stronger results hold.  The
contravariant functors $A^*_{\op}$ and $\op K^0$ are both
$\A^1$-homotopy invariants.  The cone theorem (Theorem
\ref{thm.cone}) implies the vanishing of such functors on a large
class of naturally occurring varieties including affine toric varieties
and quotients of representations of reductive groups.

Since quotients of the form $V/G$ are also the local models for good moduli
spaces of Artin stacks \cite{AHR:15} we conclude with a discussion of how the
vanishing of $A^*_{\op}(V/G)$ relates to questions about the image of
$A^*_{\op}(X)$ in $A^*\X$ when $X$ is the good moduli space of a
smooth algberaic stack $\X$.

{\bf Dedication.} It is a pleasure to dedicate this work to William Fulton on the occasion of his 80th birthday. 

\section{Homotopy invariant functors}
Fix a ground field $k$ and let $\Sch/k$ denote the category of $k$-schemes
of finite type.
\begin{defn}
  A homotopy invariant functor is  a contravariant functor $H \colon
  \Sch/k \to \Ab$ such that the pullback $H(X) \stackrel{\pi^*} \to H(X \times
  \A^1)$ is an isomorphism for all $X$ in $\Sch/k$.
  Likewise if $G$ is an algebraic group then a $G$-homotopy invariant
  functor is a contravariant functor $H^G \colon \Sch^G/k \to \Ab$ such
  that for any $G$-scheme $X$  the pullback $H^G(X)  \stackrel{\pi^*} \to H^G(X \times
  \A^1)$ is an isomorphism, where the action of $G$ on $\A^1$ is trivial.
  \end{defn}

In this paper we focus on several homotopy invariant functors -- the
operational Chow cohomology ring defined in \cite{FuMa:81} and
\cite[Chapter 17]{Ful:84} as well as its equivariant counterpart
defined in \cite{EdGr:98} and the (equivariant) operational
$K$-theory defined by Anderson and Payne in \cite{AnPa:15}.

  \subsection{Chow cohomology}
Let $X$ be a scheme.  Following \cite{Ful:84}, let $A_k(X)$ denote the group of dimension
$k$ cycle classes modulo rational equivalence, and if $X$ is
equidimensional of dimension $n$ we let $A^k(X)$ denote the group of $(n-k)$-dimensional 
cycle classes modulo rational equivalence.  If $X$ is
smooth and equidimensional, then the intersection product on $A_k(X)$ as constructed in
\cite[Chapter 6.1]{Ful:84} makes $A^*(X)$ into a commutative, graded
ring.

For general schemes  \cite[Chapter 17]{Ful:84} defines a graded operational
Chow cohomology ring
$A^*_{\op}(X) := \oplus_{k \geq 0} A^k_{\op}(X)$: an
element $c\in A^k_{\op}(X)$ is a collection of homomorphisms of groups:
\[ c^{(k)}_g: A_pX'  \rightarrow A_{p-k}X'\]
for every morphism $g:X'\rightarrow X$  which are
compatible with respect to proper pushforward, and pullbacks
along flat morphisms and regular embeddings (see \cite[Definitions 17.1 and 17.3]{Ful:84}).  The
product is given by composition and turns $A^*_{\op}(X)$ into
a graded ring called the $\tbf{Chow
  cohomology ring}$ of $X$. Moreover, if
$X$ has a resolution of singularities (e.g. if the characteristic of
the ground field is zero or if $X$ is a toric variety) then
$A^*_{\op}(X)$ is known to be commutative.
If $X$ is smooth, then \cite[Corollary
  17.4]{Ful:84} proves that the Poincar\'e duality map $A^k_{\op}(X)
\rightarrow A^k(X) = A_{n-k}(X)$ is an isomorphism of rings where 
the intersection product agrees with the product given by composition.

\begin{rem}
In \cite[Chapter 17]{Ful:84} the Chow cohomology
ring is also denoted $A^*(X)$ without the inclusion of the subscript
'op'.
\end{rem}

\begin{prop}
  Chow cohomology is an $\A^1$-homotopy invariant functor.
\end{prop}
\begin{proof}
  Consider the pullback $\pi^* \colon A^*_{\op}(X) \to A^*_{\op}(X \times \A^1)$ where $\pi \colon X \times \A^1 \to X$ is the projection.
First note that injectivity of $\pi^*$  is a formal consequence of the functoriality; the composition $X \stackrel{\iota} \to X \times \A^1 \stackrel{\pi} \to X$
  is the identity, where $\iota(x) = (x,0)$.
  
  Suppose that $c \in A^k_{\op}(X \times \A^1)$. We wish to
  show that $c = \pi^* d$ for some $d \in A^k_{\op}(X)$.
  Given a morphism $Y \stackrel{f} \to X$ let
  $g = f \times \id \colon Y \times \A^1 \to X \times \A^1$.
  Since the flat pullback $\pi^* \colon A_*(Y) \to A_*(Y \times \A^1)$
  is an isomorphism and this isomorphism is compatible with other operations
  on Chow homology we can define a class $d \in A^*_{\op}(X)$ such
  that $\pi^*d = c$ by the formula
  $d_f(\alpha) = (\pi^*)^{-1}c_g(\pi^*\alpha)$.
\end{proof}
\subsubsection{Equivariant Chow cohomology}
An equivariant version of operational Chow cohomology was defined in
\cite{EdGr:98}. An element $c \in A^k_{\op, G}(X)$ is a collection operations on equivariant Chow groups $c_f \colon
A_*^G(X') \to A_{*-k}^GG(X')$ for every equivariant morphism $X'
\stackrel{f} \to X$ compatible with equivariant proper pushforward and
equivariant flat maps and equivariant regular embeddings maps. \cite[Corollary 2]{EdGr:98} states
that if $X$ admits a resolution of singularities then
$A^k_{\op,G}(X)$ can be identified with the operational Chow
group $A^k_{\op}(X_G)$ where $X_G$ is an algebraic space of the
form $X \times^G U$. Here $U$ is an open set in a representation $V$ of
$G$ on which $G$ acts freely and $\codim (V \setminus U) >
k$.  It follows from this identification that the equivariant Chow
cohomology groups $A^k_{{\op},G}(X)$ enjoy all of the formal
properties of ordinary operational Chow cohomology. In particular, the
functor $A^*_{{\op},G}$ is a homotopy invariant functor on the
  category of schemes or algebraic spaces with a $G$-action.

\subsection{Operational $K$-theory}
Following \cite{AnPa:15}, if $X$ is a scheme 
we denote by $K_0(X)$ the Grothendieck group of coherent sheaves, and
$K^0(X)$ the Grothendieck group of perfect complexes. If $X$ has an ample
family of line bundles, then $K^0(X)$ is the same as the naive Grothendieck
group of vector bundles.

For any scheme $X$, Anderson and Payne define the {\bf operational
  $K$-theory} $\op K^0(X)$ of $X$ as follows.  An element $c \in \op
K^0(X)$ is a collection of operators $c_f \colon K_0(X') \to K_0(X')$
indexed by morphisms $X' \stackrel{f} \to X$ compatible with proper
pushforward, flat pullback and pullback along regular embeddings.

For any scheme $X$, there is a canonical map $\op K^0(X) \to K_0(X)$ given
by $c \mapsto c_{\id_X}({\mathcal O}_X)$. If $X$ is smooth, then
\cite[Corollary 4.5]{AnPa:15} states that this map is an isomorphism.

\begin{thm}\cite[Theorem 1.1] {AnPa:15}
  $\op K^0$ is a homotopy invariant functor.
\end{thm}

\subsubsection{Equivariant operational $K$-theory}
Anderson and Payne also define the equivariant operational $K$-theory
ring as the ring of operations on the equivariant Grothendieck group of
coherent sheaves, $K_0^G(X)$.  Since the equivariant Grothendieck
group is an $\A^1$-homotopy invariant the proof of \cite[Theorem 1.1
]{AnPa:15} goes through and we conclude that $\op K^0_G$ is a homotopy
invariant functor. If $G = T$ is a torus and $X$ is smooth,
then Anderson and Payne also
prove that $\op K^0_T(X)$ can be identified with $K_0^T(X)$.

\section{The cone theorem}

Fix a base scheme $X$ of finite type defined over a field $k$.
\begin{defn}
An $X$-cone is a scheme of the form $C = {\Spec}_X S$
where
$S = \oplus_{n=0}^\infty S_i$ is a finitely generated 
graded ${\mathcal O}_X$-algebra such that $S_0 = {\mathcal O}_X$.
(Note that we do not require that $S$ be locally generated in degree one.)

More generally, if $G$ is an algebraic group and $X$ is a $G$-scheme then we say that $C = \Spec S$ is a $G$-cone
if the $S_i$ are sheaves of $G$-${\mathcal O}_X$ modules and multiplication
of local sections is $G$-equivariant.
\end{defn}

The inclusion $S_0 \to S$
defines a projection $\rho \colon C \to X$ and the identification
of $S_0 = S/S^+$ defines an inclusion $\iota \colon X \to C$. Clearly,
$\rho \circ \iota = \id_X$.

The key property of homotopy invariant functors is the following
cone theorem.
\begin{thm}{\cite[cf. Exercise IV.11.5]{Wei:13}} \label{thm.cone}
  The pullbacks $\rho^*$ and $\iota^*$ are inverses. In particular
  $H(X) = H(C)$. Likewise if $H^G$ is an $G$-homotopy invariant
  functor and $C =\Spec S$ is a $G$-cone then $H^G(X) = H^G(Y)$. 
\end{thm}
\begin{proof}
  We give the proof in the non-equivariant case as the proof in
  the equivariant case is identical.
  
  Since $\rho \circ \iota = \id_X$, we know that $\iota^* \circ \rho^* \colon X \to X$ is the identity. In particular, $\rho^*$ is injective. Thus
  it suffices to prove that $\rho^* \circ \iota^* \colon C \to C$ is an
  isomorphism.
  
  Since $C$ is a cone over $X = \Spec S_0$,
  the map of graded rings  $S \to S[t]$, sending $S_i$ to $t^i S_i$.
defines an $\A^1$ action
  $\sigma \colon C \times \A^1 \to C$
with fixed scheme $X= \Spec S_0$.

Let $s_t \colon C \to C$ be the map $x \mapsto tx$. For $t \neq 0$,
$s_t$ is an isomorphism with inverse $s_{t^{-1}}$ and $s_0 \colon
C \to C$ is the composition $\iota \circ \rho$. The map
$s_t$ is itself a composition
$$C \stackrel{i_t} \to C \times \A^1 \stackrel{\sigma} \to C$$
where $i_t \colon X \hookrightarrow X \times \A^1$ is the inclusion
$x \mapsto (x,t)$.

Now if $\pi \colon X \times \A^1 \to X$ is the projection,
then for any $t$, $\pi \circ i_t = \id$. Since $\pi^*$ is assumed to be an isomorphism, $i_t^*$ must also be an isomorphism for any $t$.  Since, for
$t \neq 0$ the composite $s_t = \sigma \circ i_t$ is an isomorphism
we see that $\sigma^*$ must also be an isomorphism. Hence,
$s_0^* = (\sigma \circ i_0)^*$ is an isomorphism.
But $s_0= (\iota \circ \rho)$, so $(\iota \circ \rho)^*$ is an
isomorphism as claimed.
\end{proof}

\begin{exa}
  Let $X \subset Y$ be a closed subscheme and let $C_XY$ be the normal
  cone of $X$ in $Y$. Theorem \ref{thm.cone} implies that the pullback
  $A^*_{\op}(X) \to A^*_{\op}(C_XY)$ is an isomorphism. In particular if $X$ is
  smooth then $A^*_{\op}(C_XY)$ is identified with the Chow ring of $X$. If
  the closed embedding $X \hookrightarrow Y$  is a regular embedding (for example if
  $X$ and $Y$ are both smooth) then $C_XY$ is the
  normal bundle to $X$ in $Y$ and this identification follows from the usual
  homotopy invariance of operational Chow rings.
\end{exa}
\subsection{Cone theorem for bivariant groups}
The operational Chow and $K$-theory rings defined by Fulton--MacPherson
and Anderson--Payne are part of a more general  construction
associated to the covariant (for proper morphisms) functors $K_0$ and
$A_*$. Given a morphism of schemes $Y \to X$ the bivariant
Chow group $A^k_{\op}(Y \to X)$ is the graded abelian group consisting
of a collection of operators $c_f \colon A_*(Y') \to A_{* -k}*(X')$
for each morphism $X' \stackrel{f} \to X$ compatible with
proper pushforward and flat pullbacks and pullback and pullback along regular
embeddings, where $Y' = Y \times_X X'$.
The group $\op K^0(Y \to X)$ is defined analogously but the compatibility
is with proper pushforward, flat pullback and pullback along regular embeddings.

The groups $A^*_{\op}(Y \to X)$ and $\op K^0(Y \to X)$ are contravariant functors
on the category whose objects are morphisms of schemes $Y \to X$
and whose morphisms are cartesian diagrams
$\begin{array}{ccc}
  Y' & \to & X'\\
 \downarrow & & \downarrow\\ Y & \to  & X
\end{array}
$\\
It is easy to show that pullback along the diagram
$\begin{array}{ccc}
  Y \times \A^1  & \to & X \times \A^1\\ \downarrow & & \downarrow\\
  Y & \to & X
\end{array}
$\\
induces isomorphisms $A^*_{\op}(Y \to X) \to
A^*_{\op}(Y \times \A^1 \to X \times \A^1)$,\\ $\op K^0(Y \to X) \to
\op K^0(Y \times \A^1 \to X \times \A^1)$.
As a corollary we obtain a cone isomorphism theorem for these bivariant
groups.
\begin{cor}
  Given a morphism, $Y \to X$ and cone $C \to X$ let $C_Y \to Y$
  be the cone obtained by base change. Then the pullbacks
  $A^*_{\op}(Y \to X) \to A^*_{\op}(C_Y \to C)$
  and $\op K^0(Y \to X) \to \op K^0(C_Y \to C)$ are isomorphisms.
  \end{cor}
\section{Affine toric varieties}
\begin{thm}
  If $X = X(\sigma)$ is an affine toric variety defined by a strongly convex
  rational cone $\sigma$ in a lattice $N$, then for any homotopy invariant
  functor $H$ on $\Sch/k$,  $H(X) = H(T_0)$ where
  $T_0$ is an algebraic torus.
\end{thm}
\begin{proof}
  Since $X(\sigma)$ is an affine toric variety the proof of
  \cite[Proposition 3.3.9]{CLS:11} shows that we can decompose
  $X = X(\overline{\sigma}) \times T_0$ where $T_0$ is a torus
  and $\overline{\sigma}$ is a full dimensional cone. Since
  $\overline{\sigma}$ is full-dimensional the semi-group
  $S_{\overline{\sigma}}$ is generated in positive degree, so
  $R=k[X(\overline{\sigma})]$ is a positively graded ring with
  $R_0 = k$.

  Hence $S = k[X(\sigma)] = k[T_0] \otimes_k R$ is a positively graded ring
  with $S_0 = k[T_0]$. Hence by the cone theorem,
  $H(X) = H(T_0)$
\end{proof}

\begin{cor} \label{cor.toric}
  If $X$ is an affine toric variety then $A^0(X) = \ZZ$ and
  $A^k_{\op}(X) = 0$ for $k > 0$.
Likewise, $\op K^0(X) = \ZZ$.
\end{cor}
\begin{proof}
  By the theorem we know that $A^*_{\op}(X) = A^*_{\op}(T_0)$
  and $\op K^0(X) = \op K^0(T_0)$. Since a torus is an open
  subset of $\A^n$, $A^k_{\op}(T_0) = A^k(T_0) = 0$ if $k > 0$
  and $A^0(X) = \ZZ$ for any $X$. Likewise, $\op K^0(X) =
  \op K^0(T_0) =K_0(T_0) = \ZZ$.
\end{proof}
\begin{exa} When $X$ is a complete toric variety
  then we know that $A^k_{\op}(X) = \Hom(A_k(X), \ZZ)$. However, for
  non-complete toric variety this result fails. For example
  if $X = \A^1$ then $A_0(X) = 0$ but $A^0(X) = \ZZ$. Another example
is to let $\sigma$ denote the cone generated by
  $\{(1,0,1),(0,-1,1),(-1,0,1),(0,1,1)\}$ in $\RR^3$.
  One can compute that $A_2(X(\sigma)) = \ZZ/2 \oplus \ZZ$. Thus,
 $\Hom(A_2(X(\sigma)), \ZZ) = \ZZ$ but by
Corollary \ref{cor.toric}, $A^2_{\op}(X(\sigma)) = 0$.
\end{exa}
\begin{exa}
In the equivariant case we have an analogous result for the $T$-equivariant
operational Chow ring and $K$-theory. To simplify the notation
we assume that the cone $\sigma$ is full dimensional.
\begin{cor}
  If $X$ is an affine toric variety associated to a full dimensional cone $\sigma$
  then $A^*_{\op, T}(X) = \Sym(X(T))$ and $\op K^0_T(X) = R(T)$.
Here $\Sym(X(T))$ is the polynomial algebra generated by the character group
    of $T$ and $R(T)$ is the representation ring of $T$.
\end{cor}
\begin{proof}
  In this case $S_0 = \Spec k$, so $A^*_{\op, T}(X) = A^*_T(\pt) =\Sym(X(T))$
  by \cite[Section 3.2]{EdGr:98} and $\op K_T^0(X) =K_T^0(\pt) =R(T)$.
\end{proof}
\end{exa}

\subsection{An alternative proof of the vanishing of Chow cohomology and operational $K$-theory on affine toric varieties}

In \cite{Ric:19, EdRi:19} a more involved proof that $A^*_{\op}(X) = \op K^0(X) = \ZZ$
when $X$ is an affine toric variety is given. The proof of 
both of these statements
rests on the fact that both $A^*_{\op}(X)$ and $\op K^0(X)$ satisfy the
following descent property for proper surjective morphisms.

\begin{quote}If $X' \to X$ is a proper surjective morphism and if $H$ denotes either functor
  $A^*_{\op}$ or $\op K^0$ then the sequence
  $$0 \to H(X)\otimes \QQ \to H(X')\otimes \QQ \stackrel{p_1^* - p_2^*} \to H(X' \times_X X') \otimes \QQ$$ is
  exact where $p_1, p_2$ are the two projections $X' \times_X X' \to X'$.
\end{quote}

This descent property does not hold for arbitrary homotopy
invariant functors -- for example it need not hold for the functor $\op K^0_{G}$ when
$G$ is not a torus. However when the descent property holds for a functor
$H$, it can be used
as a tool to calculate $H$ on singular schemes \cite{Kim:92, AnPa:15}.

  \section{Operational Chow rings of good moduli spaces}
  The goal of this section is to explain how the cone theorem for homotopy
  invariant functors can shed light on questions about the structure
  of the operational Chow ring for quotients of smooth varieties and, more generally, good moduli spaces of smooth Artin stacks.

\subsection{Strong cycles on good moduli spaces of Artin stacks}
Let $G$ be a linearly algebraic group acting on a scheme $X$. We say
that a scheme $Y$ equipped with a  $G$-invariant morphism $p \colon X \to Y$
is a {\em good quotient} if $p$ is affine and $(p_*{\mathcal O}_X)^G = {\mathcal O}_Y$. The basic example is
the quotient $X^{ss} \to X^{ss}/G$ where $X^{ss}$ is the set of
semi-stable points (with respect to a choice of linearization) for the action of a linearly reductive group on a projective
variety $X$. This definition was extended to Artin stacks by Alper.
  \begin{defn}[{\cite[Definition 4.1]{Alp:13}}]
Let $\X$ be an Artin stack and let $X$ be an algebraic space. We say
that $X$ is a {\em good moduli space of $\X$} if there is a morphism
$\pi \colon \X \to X$ such that
\begin{enumerate}
\item $\pi$ is {\em cohomologically affine} meaning that the pushforward functor $\pi_*$
on the category of quasi-coherent ${\mathcal O}_\X$-modules is exact.

\item $\pi$ is {\em Stein} meaning that the natural map ${\mathcal O}_X \to \pi_* {\mathcal O}_\X$ is an isomorphism.
\end{enumerate}
\end{defn}

\begin{rem}
If $\X = [Z/G]$ where $G$ is a linearly reductive algebraic group
then the statement that $X$ is a good moduli space for $\X$ is 
equivalent to the
statement that $X$ is the good quotient of $Z$ by $G$. 
\end{rem}

\begin{defn}[{\cite{EdRy:17}}] \label{def.stablegms}
Let $\X$ be an Artin stack with good moduli space $X$ and
let $\pi \colon \X \to X$ be the good moduli space morphism. We say that a closed point $x$
of  $\X$ is
{\em
   stable} if $\pi^{-1}(\pi(x)) = x$ under the induced map of
  topological spaces $|\X| \to |X|$. A closed point $x$ of $\X$ is {\em
    properly stable} if it is stable and the stabilizer of $x$ is finite.

We say $\X$ is  stable (resp.~properly stable) if there is a good moduli
space $\pi \colon \X\to X$ and the set of stable (resp.~properly stable) points is non-empty. Likewise we say that $\pi$ is a stable (resp.~properly stable) good moduli space morphism.
\end{defn}
\begin{rem} \label{remark:ps}
Again this definition is modeled on GIT. If $G$ is a linearly reductive group
and $X^{ss}$ is the set of semistable points for a linearization of the 
action of $G$ on a projective variety $X$ then a (properly) stable point
of $[X^{ss}/G]$ corresponds to a (properly) stable orbit in the sense of GIT. 
The stack $[X^{ss}/G]$ is stable if and only if $X^{s} \neq \emptyset$. Likewise
$[X^{ss}/G]$ is properly stable if and only if $X^{ps} \neq \emptyset$. As is the case for GIT quotients, the set of stable (resp. properly stable points)
is open \cite{EdRy:17}.
\end{rem}

\begin{defn} \cite{EdSa:18, EdSa:17}
\label{def:strong}
Let $\X$ be an irreducible Artin stack with stable good moduli space $\pi \colon \X \to X$. A
closed integral substack $\Z \subseteq \X$ is {\em strong} if $\codim_\X\Z
=\codim_X\pi(\Z)$ and $\Z$ is saturated with respect to $\pi$, i.e.~$\pi^{-1}(\pi(\Z)) = \Z$ as stacks. We say $\Z$ is {\em topologically strong} if
$\codim_\X\Z=\codim_X\pi(\Z)$ and $\pi^{-1}(\pi(\Z))_{red} = \Z$.
\end{defn}

Let $A^*_{\tst}(\X/X)$ be the subgroup of $A^*(\X)$ generated by topologically strong cycles and $A^*_{\st}(\X/X)$ be the subgroup generated by by strong cycles.
(Here the Chow group $A^*(\X)$ is the Chow group defined by Kresch \cite{Kre:99}.
When $\X = [Z/G]$ it can be identified with the equivariant Chow group
$A^*_G(Z)$ of \cite{EdGr:98}.)

The main result of \cite{EdSa:17} is the following theorem which says that if
$X$ is smooth then 
any operational class $c \in A^k(X)$ can be represented by a codimension
$k$-cycle on the stack $\X$ which is saturated with respect to the good moduli space morphism $\X \to X$. 

\begin{thm} \cite[Theorem 1.1]{EdSa:17}
  Let $\X$ be a properly stable smooth Artin stack with good moduli space
  $\X \stackrel{\pi} \to X$. Then there is a pullback $\pi\colon A^*_{\op}(X)_{\QQ} \to A^*(\X)_{\QQ}$ which is injective and
  factors through the subgroup $A^*_{\tst}(\X)_\QQ$. 
\end{thm}

\cite[Example 3.24]{EdSa:17} shows that not every topologically
strong cycle is in the image of $A^*_{\op}(X)$. However, \cite[Theorem 1.7c]{EdSa:17} states that any strong lci cycle on $\X$ is
in the image of $A^*_{\op}(X)_{\QQ}$. (A cycle $\sum_i a_i [\Z_i]$
is lci if the $\Z_i$ are closed substacks of $\X$ such that the inclusion
$\Z_i \hookrightarrow \X$ is an lci morphism.)
This leads to a number of successively weaker questions about
the operational Chow rings of good moduli spaces of smooth Artin stacks.
They can be viewed as analogues for quotients of smooth varieties by reductive
groups
of Conjectures 2 and 3 of \cite{MaVe:16}.

\begin{question} \label{ques.one}
   Is the image of $A^*_{\op}(X)_\QQ$ contained in the subgroup
    of $A^*(\X)_\QQ$ generated by strong lci cycles?
\end{question}
\begin{question} \label{ques.two}
Is the image of $A^*_{\op}(X)_\QQ$ equal to the subring of $A^*(\X)_\QQ$
generated by strong lci cycles?
\end{question}
\begin{question} \label{ques.three}
Is $A^*_{\op}(X)_\QQ$ generated by Chern classes of perfect complexes
    on $X$?
\end{question}

\begin{rem}
  Note that Question \ref{ques.three}  is an analogue of the question raised
  by Anderson and Payne about the surjectivity of the map $K^0(X) \to
  \op K^0(X)$ where $K^0(X)$ is the Grothendieck group of perfect complexes.
  Anderson and Payne prove that for 3-dimensional complete toric varieties this map is in fact surjective. By comparison in \cite{EdSa:17} the authors prove that
  $A^*(X)_\QQ = A^*_{\st}(\X/X)_\QQ$ and in the case of 3-dimensional toric varieties \cite{Ric:19} shows that $A^*(X)_\QQ$ is generated by strong lci cycles.
  \end{rem}
  Given the relation between the operational Chow ring and strong cycles leads to the following additional question.
  \begin{question}
    Is $A^*_{\st}(\X/X)$ (resp. $A^*_{tst}(\X/X)$) a subring of $A^*(\X)$; i.e.,
    is the product of strong (resp. topologically strong) cycles strong?
  \end{question}

  \subsection{The cone theorem and local models for good moduli}
  The \'etale slice theorem of Alper, Hall and Rydh \cite{AHR:15} states if $\X$
  is a smooth stack
  then  at a closed point $x$ of $\X$ the  good moduli morphism $\X \to X$ is \'etale locally isomorphic to the quotient $[V/G_x] \to \Spec k[V]^{G_x}$
  where $G_x$ is the inertia group of $x$ in $\X$ and $V$ is a representation
  of $G_x$. Thus we may view stacks of the form $[V/G]$ with their
  good moduli spaces $\Spec k[V]^G$ as the ``affine models''
  of smooth stacks with good moduli spaces.

  The cone theorem has the following corollary.
\begin{cor} \label{cor.invring}
  Let $V$ be a representation of a reductive group $G$
  and let $X = \Spec k[V]^G$ be the quotient. Then for any homotopy invariant
  functor $H(X) = H(\Spec k)$.
  In particular, $A^*_{\op}(X) = \ZZ$ and $\op K^0(X) = \ZZ$.
  \end{cor}
\begin{proof} 
  Since $G$ acts linearly on $V$ the action of $G$ on $k[V]$
  preserves the natural grading. Hence the invariant ring $k[V]^G$ is
  also graded. Thus by Theorem \ref{thm.cone} $H(X) = H(\Spec k)$
  since $k$ is the 0-th graded piece of $k[V]^G$.
\end{proof}

Combining Corollary \ref{cor.invring} with \cite[Theorem 1.7c]{EdSa:17} yields the following result.
\begin{cor}
  If $\Z \subset \X = [V/G]$ is a proper closed substack which is strongly regularly embedded then $[\Z]$ is torsion in $A^*(\X) = A^*_G(\pt)$. In particular
  if $A^*_G(\pt)$ is torsion free (for example if $G$ is torus or $\GL_n$)
  then $[\Z] = 0$.
\end{cor}

Note that any integral strong divisor ${\mathcal D}$ on $\X = [V/G]$ is necessarily defined by a single
$G$-invariant equation. In this case ${\mathcal O}({\mathcal D})$ is
an equivariantly trivial line bundle so $[{\mathcal D}] = c_1({\mathcal O}({\mathcal D})) = 0$ in the equivariant Chow ring $A^*_G(V)$. More generally,
if an integral substack $\Z \subset [V/G]$ is a complete intersection of strong divisors, then its class in $A^*_G(V)$ is also 0. In particular it shows that
in the ``local case'' the image of $A^*_{\op}(X)$ is contained in the
subgroup generated by strong global complete intersections.

This leads to the following question.
\begin{question}
  Are there examples of strong integral substacks $\Z \subset [V/G]$
  such that $[\Z] \neq 0$ in $A^*([V/G])_\QQ$?
  \end{question}

\subsection{Acknowledgments} The authors are grateful to Sam Payne and Angelo Vistoli for helpful comments.

\bibliographystyle{amsmath}
\def\cprime{$'$} \def\cprime{$'$} \def\cprime{$'$}

\end{document}